\documentclass[twoside]{article}

\usepackage{tikz}

\usepackage{amssymb,amsmath,amsthm}
\usepackage[all]{xy}

\usepackage{tabls}
\usepackage{booktabs}
\usepackage{hhline}

\usepackage{ulem,color}

\newcommand{\C}{{\mathbb {C}}}
\newcommand{\R}{{\mathbb {R}}}
\newcommand{\Q}{{\mathbb {Q}}}
\newcommand{\Z}{{\mathbb {Z}}}

\newcommand{\F}{{\mathcal {F}}}
\newcommand{\ep}{\varepsilon}
\newcommand{\re}{\rho^\varepsilon}
\newcommand{\sm}{\setminus}
\newcommand{\Hom}{{\rm Hom}}
\newcommand{\suchthat}{\mathrel{|}}

\newcommand{\ii} [1] [n] {{}^{I_{#1}}_{I'_{#1}} } 
\newcommand{\iii} [2] {{}^{(I_{n-1},{#1})}_{(I'_{n-1},{#2})} } 
\renewcommand{\S}{\mathfrak{S}}

\makeatletter
 
 \@addtoreset{equation}{section}
\makeatother

\newtheorem{thm}{Theorem}[section]
\newtheorem{lem}[thm]{Lemma}
\newtheorem{prop}{Proposition}[section]
\newtheorem{cor}{Corollary}[section]

\theoremstyle{definition}
\newtheorem{dfn}{Definition}[section]

\theoremstyle{remark}

\theoremstyle{nonumber}
\newtheorem{claim}{Claim}

\title{The $T$-equivariant Integral Cohomology Ring of $F_4/T$}
\author{Takashi SATO \thanks{\texttt{t-sato@math.kyoto-u.ac.jp}}}
\date{\empty}
\begin{document}

\maketitle

\begin{abstract}
We determine the $T$-equivariant integral cohomology of $F_4/T$ combinatorially by the GKM theory,
where $T$ is a maximal torus of the exceptional Lie group $F_4$ and acts on $F_4/T$ by the left multiplication.
\end{abstract}

\section{\textsc{Introduction and statement of the result}}
\label{intro}
Let $G$ be a compact, connected Lie group and $T$ its maximal torus.
The homogeneous space $G/T$ is a flag variety and it plays an important role in topology, algebraic geometry, representation theory, and combinatorics.
In particular, the $T$-equivariant integral cohomology ring $H^*_T(G/T) = H^*(ET \times_T G/T)$ is especially important,
where $T$ acts on $G/T$ by the left multiplication.

Goresky, Kottwitz, and MacPherson \cite{GKM} gave a powerful method to determine the equivariant cohomology with $\Q$ coefficients of some good spaces.
It is called the GKM theory.
Let us explain how the GKM theory works in our situation.
Since the fixed points set $(G/T)^T$ is identified with the Weyl group $W(G)$,
the inclusion $i \colon (G/T)^T \to G/T$ induces the map
\[
i^* \colon H^*_T(G/T) \to H^*_T((G/T)^T) = \prod_{W(G)} H^*(BT) = {\rm Map}(W(G),H^*(BT)).
\]
Tensoring with $\Q$, $i^*$ is injective by the localization theorem (cf. \cite[Theorem (III.1)]{H}).
The GKM theory gives a way to describe the image of this map $i^*$, which is restated by Guillemin and Zara \cite{GZ} as follows.
The image of $i^*$ is completely determined by a graph with additional data obtained from $G$.
Precisely they defined the ``cohomology" ring of the graph as a subring of ${\rm Map}(W(G),H^*(BT))$ and showed that it coincides with the image of $i^*$.
This graph is called a GKM graph.
Harada, Henriques, and Holm \cite{HHH} showed that $i^*$ is injective with integer coefficient when $G$ is simple and is not of type $C$.

By concrete computations by the GKM theory,
for a simple Lie group $G$ of classical types and of type $G_2$,
Fukukawa, Ishida, and Masuda \cite{FIM}, \cite{F} determined the cohomology ring of the GKM graph of $G/T$.
Hence they determined the equivariant integral cohomology ring $H^*_T(G/T)$ for a Lie group $G$ of type $A$, $B$, $D$, and $G_2$.
In this paper we determine the $T$-equivariant integral cohomology ring of $F_4/T$ by the GKM theory.

For $x=(x_1, \ldots, x_n)$, let $e_i(x)$ denote the $i^{\text{th}}$ elementary symmetric polynomial in $x_1, \ldots, x_n$.
Put $x^k=(x_1^k, \ldots, x_n^k)$.
For a linear transformation $\alpha$ of $\R x_1 \oplus \cdots \oplus \R x_n$, let $\alpha x = (\alpha x_1, \ldots, \alpha x_n)$.
Then $e_i(x^k)$ and $e_i(\alpha x)$ denote
the $i^{\text{th}}$ elementary symmetric polynomial in $x_1^k, \ldots, x_n^k$ and $\alpha x_1, \ldots, \alpha x_n$, respectively.
The following theorem is the main result of this paper.
In this theorem
$t=(t_1, t_2, t_3, t_4)$, $\tau=(\tau_1, \tau_2, \tau_3, \tau_4)$,
and $\rho$ is the linear transformation of $\R t_1 \oplus \cdots \oplus \R t_4$ defined as \eqref{rho}.
\begin{thm}
\label{main}
Let $T$ be a maximal torus of $F_4$ which acts on $F_4/T$ by the left multiplication,
then the $T$-equivariant integral cohomology ring of $F_4/T$ is given as:
\[
H^*_T(F_4/T) \cong \Z[t_i, \gamma, \tau_i, \gamma_i, \omega \suchthat 1 \leq i \leq 4]/(r_1', R_i, r_{2i}, r_{12} \suchthat 1 \leq i \leq 4),
\]
where $|t_i| = |\gamma| = |\tau_i| = 2$, $|\gamma_i| = 2i$, $|\omega|=8$,
\begin{align*}
&r_1' = e_1(t) -2\gamma, 							&& R_i = e_i(\tau) -e_i(t) -2\gamma_i \:\: (i=1,2,3),\\
&r_{12} = \omega(\omega-e_4(\rho t))(\omega+e_4(\rho^2 t)),		&& R_4 = e_4(\tau)-e_4(t) -2\gamma_4 -\omega,\\
&r_2 = \sum_{j=1}^{2} (-1)^j \gamma_j(\gamma_{2-j} +e_{2-j}(t)), 	&& r_4 = \sum_{j=1}^{4} (-1)^j \gamma_j(\gamma_{4-j} +e_{4-j}(t)) -\omega,\\
&r_6 = \sum_{j=2}^{4} (-1)^j \gamma_j(\gamma_{6-j} +e_{6-j}(t)) +(\gamma_2 + \gamma^2) \omega, && r_8 = \gamma_4(\gamma_{4} +e_{4}(t)) +\omega^2 + (\gamma_4 -e_4(\rho t))\omega.
\end{align*}
\end{thm}

The ordinary integral cohomology ring $H^*(F_4/T)$ was determined by Toda and Watanabe \cite{TW}.
We can obtain the integral cohomology ring of $F_4/T$ as a corollary of Theorem \ref{main} as follows.
There is a fibration sequence
\[
  \xymatrix{
   F_4/T \ar[r] & ET \times_T F_4/T \ar[r]^(0.65){p} & BT.
  }
\]
Since the projection $p \colon ET \times_T F_4/T \to BT$ restricts to $p \circ i \colon ET \times_T (F_4/T)^T \to BT$,
where $i$ is the inclusion $ET \times_T (F_4/T)^T \to ET \times_T F_4/T$,
the induced map $(p \circ i)^* \colon H^*(BT) \to H^*(ET \times_T (F_4/T)^T) = {\rm Map}(W(F_4),H^*(BT))$
sends elements of $H^*(BT)$ to constant functions.
In Theorem \ref{main}, $t_1$, $t_2$, $t_3$, $t_4$, and $\gamma$ correspond to constant functions (see Section \ref{sec. pf. of thm. main}).
Since the cohomology of $F_4/T$ and $BT$ have vanishing odd parts,
the Serre spectral sequence of the fibration $p$ collapses at the $E_2$-term.
Hence $H^*(F_4/T) \cong H^*_T(F_4/T)/(t_1, t_2, t_3, t_4, \gamma)$.

\begin{cor}[{\cite[Theorem A]{TW}}]
\label{F4/T}
The integral cohomology ring of $F_4/T$ is given as:
\[
H^*(F_4/T) \cong \Z[\tau_i, \gamma_1, \gamma_3, \omega \suchthat1 \leq i \leq 4]/(\overline{r}_1, \overline{r}_2, \overline{r}_3, \overline{r}_4, \overline{r}_6, \overline{r}_8, \overline{r}_{12}),
\]
where
\begin{align*}
&\overline{r}_1 = 2\gamma_1 - e_1(\tau), &&\overline{r}_2 = 2\gamma_1^2 - e_2(\tau),\\
&\overline{r}_3 = 2\gamma_3 - e_3(\tau), &&\overline{r}_4 = e_4(\tau) -2\gamma_1e_3(\tau) +2\gamma_1^4 -3\omega,\\
&\overline{r}_6 = -\gamma_1^2e_4(\tau) +\gamma_3^2, &&\overline{r}_8 = 3e_4(\tau)\gamma_1^4 -\gamma_1^8 +3\omega(\omega+e_3(\tau)\gamma_1),\\
&\overline{r}_{12} = \omega^3.
\end{align*}
\end{cor}

Corollary \ref{F4/T} will be proved in Section \ref{pf of Cor F4/T}.
Throughout this paper, all cohomology groups and rings will be taken with integer coefficient.

\section{\textsc{GKM graph and its cohomology}}

Let $G$ be a compact connected Lie group and let $T$ be its maximal torus.
Specializing and abstracting the work of Goresky, Kottwitz, and MacPherson \cite{GKM},
Guillemin and Zara \cite{GZ} introduced a certain graph to each of whose edge an element of $H^2(BT)$ is given
and showed the $T$-equivariant cohomology of $G/T$ with complex coefficient is recovered from this graph.
Let us introduce this special graph.
Recall that there is a natural identification
\[
\Hom(T,S^1) \cong H^2(BT),
\]
 where the left hand side is the set of weights of $G$.
Let $W(G)$ and $\Phi(G)$ denote the Weyl group and the root system of $G$, respectively.
Since every root is a weight, we regard $\Phi(G) \subset H^2(BT)$.
There is a canonical action of the Weyl group $W(G)$ on $\Hom(T,S^1)$ and it restricts to $\Phi(G)$.
We denote this action as $w\alpha$ for $w \in W(G)$ and $\alpha \in H^2(BT)$.
Recall that to each $\alpha \in \Phi(G)$, one can assign a reflection $\sigma_\alpha$ which is an element of the Weyl group $W(G)$.

\begin{dfn}
The GKM graph of $G/T$ is the Cayley graph of $W(G)$ with respect to a generating set $\{\sigma_\alpha \in W(G)\suchthat\alpha \in \Phi(G)\}$
which is equipped with the cohomology classes $\pm w\alpha \in H^2(BT)$ to the edge $ww'$ satisfying $w'=w\sigma_\alpha$.
We call $\pm w\alpha$ the label of the edge $ww'$.
\end{dfn}
The ambiguity of the sign of the label $\pm w \alpha$ occurs from the equation $w' \alpha = w \sigma_\alpha \alpha = -w \alpha$.
Let us introduce the cohomology of the GKM graph.
Consider a function $f : W(G) \to H^*(BT)$ between sets.
We say that $f$ satisfies the GKM condition or $f$ is a GKM function if for any $w \in W(G)$ and $\alpha \in \Phi(G)$,
\[
f(w) - f(w\sigma_\alpha) \in (w\alpha) \subset H^*(BT),
\]
where $(x_1, \ldots, x_n)$ means the ideal generated by $x_1, \ldots, x_n$.
It is easy to see that all GKM functions form a subring of $\prod_{W(G)} H^*(BT)$,
where we identify the set of all functions $W(G) \to H^*(BT)$ with $\prod_{W(G)} H^*(BT)$.
Since the GKM graph of $G/T$ has $W(G)$ as its vertex set, a GKM function assigns an element of $H^*(BT)$ to each vertex of the GKM graph.

\begin{dfn}
Let $\mathcal{G}$ be the GKM graph of $G/T$.
The cohomology ring $H^*(\mathcal{G})$ is defined as the subring of $\prod_{W(G)} H^*(BT)$
consisting of all GKM functions.
\end{dfn}

Guillemin and Zara \cite[Theorem 1.7.3]{GZ} restated an important theorem of the GKM theory as
\[
H^*_T(G/T; \C) \cong H^*(\mathcal{G}) \otimes \C.
\]
Harada, Henriques, and Holm refined this result to the integral cohomology.
More precisely, we have:
\begin{thm}[{\cite[Theorem 3.1 and Lemma 5.2]{HHH}}]
\label{HHH}
Suppose the Lie group $G$ is simple and let $\mathcal{G}$ be the GKM graph of $G/T$.
If $G$ is not of type $C$, then there is an isomorphism
\[
H^*_T(G/T) \cong H^*(\mathcal{G}).
\]
\end{thm}

\section{\textsc{The GKM graph of $F_4/T$}}
\label{sec. GKM graph}

In this section we describe and analyze the GKM graph of $F_4/T$.
First of all let us choose a maximal torus of $F_4$.
Let $T^4$ be the standard maximal torus of $SO(9)$
and let $\overline{t}_1$, $\overline{t}_2$, $\overline{t}_3$, $\overline{t}_4 \in H^2(BT^4)$ be the canonical basis.
For the universal covering $\mu:Spin(9) \to SO(9)$ let $T = \mu^{-1}(T^4)$.
Then $T$ is a maximal torus of $Spin(9)$.
Since $Spin(9)$ is a Lie subgroup of $F_4$ (cf. \cite[Chapter 8,9,14]{A}), $T$ is also a maximal torus of $F_4$.
We fix a maximal torus of $F_4$ to $T$.
Let $t_i$ denote $\mu^*(\overline{t}_i) \in H^2(BT)$.
By definition we have
\[
H^*(BT) = \Z[t_1, t_2, t_3, t_4, \gamma]/(2\gamma - e_1(t)).
\]

To describe the Weyl group $W(F_4)$ we start with the root system of $F_4$.
The root system $\Phi(F_4)$ is given as:
\[
\Phi(F_4)=\{\pm(t_i+t_j), \pm(t_i-t_j), \pm t_k, \frac{1}{2}(\pm t_1 \pm t_2 \pm t_3 \pm t_4) \suchthat 1\leq i < j \leq 4,\ 1 \leq k \leq 4\}
\]
The roots $\pm(t_i+t_j)$ and $\pm(t_i-t_j)$ are called long roots,
and $\pm t_k$ and $\frac{1}{2}(\pm t_1 \pm t_2 \pm t_3 \pm t_4)$ are called short roots.
Put
\begin{align*}
&\alpha_1 = t_2 -t_3,	&&\alpha_2 = t_3 -t_4,\\
&\alpha_3 = t_4,		&&\alpha_4 = \frac{1}{2}(t_1 -t_2 -t_3 -t_4).\\
\end{align*}
Then the Dynkin diagram of $F_4$ is as:
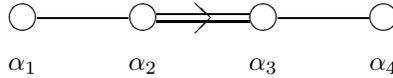
\[
    \begin{xy}
    ( -16, 0)="1"*+[Fo:<2mm>]{},
    (   0, 0)="2"*+[Fo:<2mm>]{},
    (  16, 0)="3"*+[Fo:<2mm>]{},
    (  32, 0)="4"*+[Fo:<2mm>]{},
    (   9, 0)="5"*{},
    (   7, 2)="6"*{},
    (   7,-2)="7"*{},
    ( -16,-6.4)="11"*{\alpha_1},
    (   0,-6.4)="22"*{\alpha_2},
    (  16,-6.4)="33"*{\alpha_3},
    (  32,-6.4)="44"*{\alpha_4},
    \ar@{-} "1"+/r2mm/;"2"+/l2mm/,
    \ar@{-} "2"+/r1.9mm/+/u0.5mm/;"3"+/l1.9mm/+/u0.5mm/,
    \ar@{-} "2"+/r1.9mm/+/d0.5mm/;"3"+/l1.9mm/+/d0.5mm/,
    \ar@{-} "3"+/r2mm/;"4"+/l2mm/,
    \ar@{-} "5";"6",
    \ar@{-} "5";"7",
    \end{xy}
\]
Then $W(F_4)$ is generated by the reflections $\sigma_{\alpha_i}$ for $i = 1,2,3,4$.
Since $Spin(8)$ is a Lie subgroup of $F_4$,
the root system of $Spin(8)$ is contained in $\Phi(F_4)$,
which is given as:
\[
\Phi(Spin(8))=\{\pm(t_i+t_j), \pm(t_i-t_j)\suchthat 1\leq i < j \leq 4\}
\]
It consists of all the long roots of the root system $\Phi(F_4)$.
Then the Weyl group $W(Spin(8))$ is generated by the reflections associated with the long roots,
and $W(Spin(8))$ is a subgroup of $W(F_4)$.

Put $W=W(Spin(8))$.
The vertex set $W(F_4)$ of the GKM graph of $F_4/T$ is decomposed into $6$ cosets by the next theorem.
\begin{thm}[{\cite [Theorem 14.2]{A}}]
The Weyl group $W$ of $Spin(8)$ is a normal subgroup of $W(F_4)$ and there is an isomorphism $W(F_4)/W \cong \S_3$,
where $\S_n$ is the symmetric group on $n$-letters.
Moreover $W(F_4)/W$ permutes the three root pairs
\begin{equation}
\label{rootpairs}
\pm \frac{1}{2}(t_1 +t_2 +t_3 -t_4), \quad \pm \frac{1}{2}(t_1 +t_2 +t_3 +t_4), \quad \pm t_4.
\end{equation}
\end{thm}

Let us describe the representatives of $W(F_4)/W$.
First we define an element $\rho$ of $W(F_4)$ as
\begin{equation}
\label{rho}
\rho = \sigma_{\alpha_3} \sigma_{\alpha_2} \sigma_{\alpha_1} \sigma_{\alpha_0} \sigma_{\alpha_3} \sigma_{\alpha_2} \sigma_{\alpha_1} \sigma_{\alpha_3} \sigma_{\alpha_2} \sigma_{\alpha_4},
\end{equation}
where $\alpha_0$ denotes the root $t_1-t_2$ of $Spin(8)$.
By a straightforward calculation,
we have
\begin{equation}
\label{rho t_i}
\rho t_i=
\begin{cases}
-\gamma +t_i	&(i=1,2,3)\\
\gamma -t_4		&(i=4),
\end{cases}
\end{equation}
\[
\rho^2 t_i=
\begin{cases}
-\gamma +t_4 +t_i	&(i=1,2,3)\\
-\gamma		&(i=4),
\end{cases}
\]
and
\[
\rho^3 = {\rm id}.
\]
By the above equations the root system $\Phi(F_4)$ can be rewritten as:
\[
\Phi(F_4)=\{\pm(t_i+t_j), \pm(t_i-t_j), \pm \re t_k \suchthat 1\leq i < j \leq 4,\ 1 \leq k \leq 4,\ 0\leq \ep \leq 2\}
\]
Note that $\rho$ permutes the three root pairs \eqref{rootpairs} cyclically
and $\kappa = \sigma_{t_4}$ interchanges $\pm \frac{1}{2}(t_1 +t_2 +t_3 -t_4) = \pm \rho t_4$ and $\pm \frac{1}{2}(t_1 +t_2 +t_3 +t_4) = \pm \rho^2 t_4$.
Hence $W(F_4)/W \cong \S_3$ is generated by $\rho$ and $\kappa$.
Since the equation
\begin{equation}
\label{kr=r^2k}
\kappa \rho = \rho^2 \kappa
\end{equation}
holds,
we have a coset decomposition
\[
W(F_4) = \coprod_{\substack{\ep=0,1,2 \\ \delta=0,1}} \re \kappa^\delta W.
\]

We will describe the GKM graph $\mathcal{F}_4$ of $F_4/T$.
There are $24 (= \# \Phi (F_4)/2)$ edges out of each vertex of $\mathcal{F}_4$.
The half of these edges correspond to the long roots $\pm (t_i \pm t_j)$
and the other half correspond to the short roots $\pm \re t_i$.

The subgraph induced by $W$ is the GKM graph $\mathcal{G}$ of $Spin(8)/T$
and it is well understood in \cite{FIM}.
Let $\re \kappa^\delta \mathcal{G}$ be the GKM subgraph induced by $\re \kappa^\delta W$ for $\ep=0,1,2$ and $\delta=0,1$.
For any $\ep$ and $\delta$ the induced subgraph $\re \kappa^\delta \mathcal{G}$ is isomorphic to $\mathcal{G}$ as graphs.
Indeed if an edge $ww'$ in $\mathcal{G}$ satisfies $w'=w\sigma_\alpha$ for a root $\alpha$ of $Spin(8)$,
then $\re \kappa^\delta w$ and $\re \kappa^\delta w'$ satisfy $\re \kappa^\delta w' = \re \kappa^\delta w \sigma_\alpha$,
and vice versa.
Moreover, labels of edges of $\re \kappa^\delta \mathcal{G}$ are also determined by $\mathcal{G}$ as follows.
When an edge $ww'$ has a root $\pm \beta$ as its label,
the label of the edge connecting $\re \kappa^\delta w$ and $\re \kappa^\delta w'$ is $\pm \re \kappa^\delta \beta$.
Remark that if an edge $ww'$ in $\re \kappa^\delta \mathcal{G}$ satisfies $w'=w\sigma_\alpha$, then $\alpha$ is one of the long roots.

From the above argument, it is sufficient to consider the edges connecting two of $\re \kappa^\delta \mathcal{G}$'s,
which correspond to the short roots.
Easy calculations show that
\[
\sigma_{t_4} 	  = \kappa,\quad
\sigma_{\rho t_4}   = \rho^2 \kappa,\quad
\sigma_{\rho^2 t_4} = \rho \kappa.
\]
Then the GKM graph $\mathcal{F}_4$ has an induced subgraph below,
where $e$ denotes the unit element of $W(F_4)$ and an element of $W(F_4)$ in each circle denotes a vertex of $\mathcal{F}_4$.
The labels are calculated later.
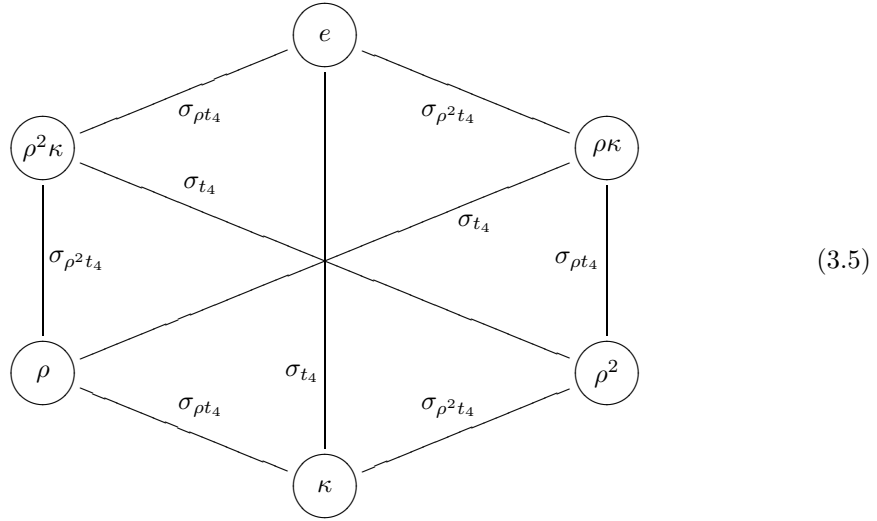
\begin{equation}
\label{subgraph}
    \begin{xy}
    (  0  , 30)="+" *+[Fo:<13pt>]{e},
    ( 37.5, 15)="1-"*+[Fo:<13pt>]{\rho \kappa},
    ( 37.5,-15)="2+"*+[Fo:<13pt>]{\rho^2},
    (  0  ,-30)="-" *+[Fo:<13pt>]{\kappa},
    (-37.5,-15)="1+"*+[Fo:<13pt>]{\rho},
    (-37.5, 15)="2-"*+[Fo:<13pt>]{\rho^2 \kappa},
    ( -3  ,-15)="a" *{\sigma_{t_4}},
    ( 33.6,  0)="a''"*{\sigma_{\rho t_4}},
    (-33.0,  0)="a'"*{\sigma_{\rho^2 t_4}},
    (-16.5, 19.5)="b" *{\sigma_{\rho t_4}},
    ( 16.5,-19.5)="b''"*{\sigma_{\rho^2 t_4}},
    ( 20  , 5)="b'"*{\sigma_{t_4}},
    (-16.5,-19.5)="c" *{\sigma_{\rho t_4}},
    (-16.5, 10)="c''"*{\sigma_{t_4}},
    ( 16.5, 19.5)="c'"*{\sigma_{\rho^2 t_4}},
    \ar@{-} "+"+/d5mm/;"-"+/u5mm/,
    \ar@{-} "1+"+/u2mm/+/r5mm/;"1-"+/d2mm/+/l5mm/,
    \ar@{-} "2+"+/u2mm/+/l5mm/;"2-"+/d2mm/+/r5mm/,
    \ar@{-} "+"+/d2mm/+/r5mm/;"1-"+/u2mm/+/l5mm/,
    \ar@{-} "1-"+/d5mm/;"2+"+/u5mm/,
    \ar@{-} "2+"+/d2mm/+/l5mm/;"-"+/u2mm/+/r5mm/,
    \ar@{-} "-"+/u2mm/+/l5mm/;"1+"+/d2mm/+/r5mm/,
    \ar@{-} "1+"+/u5mm/;"2-"+/d5mm/,
    \ar@{-} "2-"+/u2mm/+/r5mm/;"+"+/d2mm/+/l5mm/,
    \end{xy}
\end{equation}
We will calculate the reflection $\sigma_\alpha$ for a short root $\alpha$ to describe $\mathcal{F}_4$.
For example let us consider the short root $\rho t_1$ and the reflection $\sigma_{\rho t_1}$.
By \eqref{rho t_i} we have $\rho t_1 = \frac{1}{2}(t_1 -t_2 -t_3 -t_4) = \sigma_{t_2}\sigma_{t_3} (\rho t_4)$.
Then $\sigma_{\rho t_1} = \sigma_{t_2}\sigma_{t_3}\sigma_{\rho t_4}\sigma_{t_3}\sigma_{t_2}$ and $\sigma_{t_2}\sigma_{t_3} \in W$.
Since $W$ is a normal subgroup of $W(F_4)$,
we have $W \cdot \rho^2 \kappa W = \rho^2 \kappa W$ in $W(F_4)/W$.
Hence $\sigma_{\rho t_1}$ is also contained in $\rho^2 \kappa W$.
For any $i$, it is shown similarly that
\[
\sigma_{\rho t_i} \in \rho^2 \kappa W, \quad \sigma_{\rho^2 t_i} \in \rho \kappa W
\]
and obviously we have
\[
\sigma_{t_i} \in \kappa W.
\]
Hence, for any $0 \leq \ep,\ep' \leq 2$ and $\delta=0,1$,
it is independent from the choice of $i$ and $w \in \rho^{\epsilon'} \kappa^\delta W$ which coset contains $w \sigma_{\re t_i}$.

Let us calculate the label of the edge connecting the vertices $\kappa$ and $\rho$ in the GKM subgraph \eqref{subgraph},
which corresponds to a short root $\rho t_4$.
The label of the edge turns out to be $\pm \kappa(\rho t_4)$.
It follows from the relation \eqref{kr=r^2k} that
\[
\pm \kappa (\rho t_4) = \pm \rho^2 \kappa t_4 = \pm \rho^2 t_4.
\]
One can make similar calculations of the labels of other edges in the GKM subgraph \eqref{subgraph}.
For any $w \in W$, $w$ fixes three sets of short roots $\{\pm t_i\}_{i=1}^4$, $\{\pm \rho t_i\}_{i=1}^4$ and $\{\pm \rho^2 t_i\}_{i=1}^4$
since $w$ permutes $t_i$'s and changes the signs of even number of $t_i$'s.
Hence the label $\pm \re \kappa^\delta w (\alpha)$ is calculated similarly for any short root $\alpha$.

We can now describe a schematic diagram of $\mathcal{F}_4$ as below.
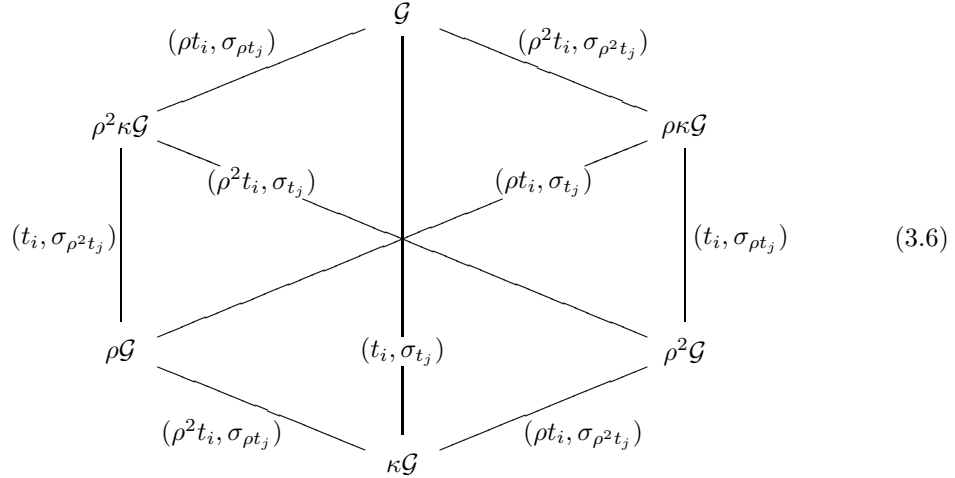
\begin{equation}
\label{sch. diag. F_4}
    \begin{xy}
    (  0  , 30)="+" *{\mathcal{G}},
    ( 37.5, 15)="1-"*{\rho \kappa \mathcal{G}},
    ( 37.5,-15)="2+"*{\rho^2 \mathcal{G}},
    (  0  ,-30)="-" *{\kappa \mathcal{G}},
    (-37.5,-15)="1+"*{\rho \mathcal{G}},
    (-37.5, 15)="2-"*{\rho^2 \kappa \mathcal{G}},
    (  0  ,-15)="+-" *{(t_i,\sigma_{t_j})},
    ( 18.75,  7.5)="1+1-"*{(\rho t_i,\sigma_{t_j})},
    (-18.75,  7.5)="2+2-"*{(\rho^2 t_i,\sigma_{t_j})},
    (-24, 26)="+2-" *{(\rho t_i,\sigma_{\rho t_j})},
    (-24,-26)="-1+" *{(\rho^2 t_i,\sigma_{\rho t_j})},
    ( 45,  0)="1-2+"*{(t_i,\sigma_{\rho t_j})},
    (-45,  0)="1+2-"*{(t_i,\sigma_{\rho^2 t_j})},
    ( 24,-26)="-2+"*{(\rho t_i,\sigma_{\rho^2 t_j})},
    ( 24, 26)="+1-"*{(\rho^2 t_i,\sigma_{\rho^2 t_j})},
    \ar@{-} "+"+/d3mm/;"+-"+/u2mm/,					\ar@{-}"+-"+/d2mm/;"-"+/u4mm/,
    \ar@{-} "1+"+/u2mm/+/r5mm/;"1+1-"+/d2.2mm/+/l5.5mm/,	\ar@{-} "1+1-"+/u2.2mm/+/r5.5mm/;"1-"+/d2mm/+/l5mm/,
    \ar@{-} "2+"+/u2mm/+/l5mm/;"2+2-"+/d2.2mm/+/r5.5mm/,	\ar@{-} "2+2-"+/u2.2mm/+/l5.5mm/;"2-"+/d2mm/+/r5.mm/,
    \ar@{-} "+"+/d2mm/+/r5mm/;"1-"+/u2mm/+/l5mm/,
    \ar@{-} "1-"+/d3mm/;"2+"+/u4mm/,
    \ar@{-} "2+"+/d2mm/+/l5mm/;"-"+/u2mm/+/r5mm/,
    \ar@{-} "-"+/u2mm/+/l5mm/;"1+"+/d2mm/+/r5mm/,
    \ar@{-} "1+"+/u4mm/;"2-"+/d3mm/,
    \ar@{-} "2-"+/u2mm/+/r5mm/;"+"+/d2mm/+/l5mm/,
    \end{xy}
\end{equation}
This diagram means the followings.
For example, $\mathcal{G}$ and $\rho \mathcal{G}$ are not adjacent in this diagram.
It means that for any vertices $w \in W$ and $w' \in \rho W$,
they are not adjacent.
On the other hand, $\rho \mathcal{G}$ and $\rho \kappa \mathcal{G}$ are adjacent in this diagram,
and a pair $(\rho t_i,\sigma_{t_j})$ is assigned to the edge.
The first entry $\rho t_i$ is a root and the second entry $\sigma_{t_j}$ is a reflection.
If two vertices $w \in \rho W$ and $w' \in \rho \kappa W$ are adjacent in $\mathcal{F}_4$,
then they satisfy $w'=w\sigma_{t_j}$ for some $j$,
and the edge $ww'$ is labeled by $\rho t_i$ for some $i$.
The label $\pm \rho t_i$ equals to $\pm w t_j$.
Especially each vertex of $\rho \mathcal{G}$ is connected to $4$ vertices of $\rho \kappa \mathcal{G}$
by the edges correspond to the short roots $t_j$'s $(1 \leq j \leq 4)$, and vice versa.
The labels of these edges are $\pm \rho t_i$'s $(1 \leq i \leq 4)$.
Every $\rho t_i$'s appear as the labels of the edges out of each vertex of $\rho \mathcal{G}$.
The situation is the same for any connected two subgraphs in the schematic diagram \eqref{sch. diag. F_4}.

\section{\textsc{Proof of the main theorem}}
\label{sec. pf. of thm. main}
There is a fibration sequence
\begin{equation}
\label{fib. seq.}
  \xymatrix{
   F_4/T \ar[r] & ET \times_T F_4/T \ar[r] & BT.
  }
\end{equation}
The cohomology rings of $F_4/T$ and $BT$ are free as $\Z$-modules and have vanishing odd parts.
As shown in Section \ref{sec. GKM graph}, $H^*(BT)$ has five generators $t_1$, $t_2$, $t_3$, $t_4$, and $\gamma$ of degree $2$
with one relation of degree $2$.
According to \cite{TW},
$H^*(F_4/T)$ has $\tau_1$, $\tau_2$, $\tau_3$, $\tau_4$, and $\gamma_1$ of degree $2$,
$\gamma_3$ of degree $6$, and $\omega$ of degree $8$ as its generators,
and $H^*(F_4/T)$ has seven relations of degree $2$, $4$, $6$, $8$, $12$, $16$, and $24$.
We can expect $H^*_T(F_4/T)$ has corresponding generators and relations.
It is easy to see the Poincar\'e series of $F_4/T$ and $BT$ are
\[
(1+x^8+x^{16})\prod_{i=1}^4 \frac{1-x^{4i}}{1-x^2} \quad \text{and} \quad \frac{1}{(1-x^2)^4},
\]
respectively.
Hence we obtain the following proposition by the Serre spectral sequence for \eqref{fib. seq.}.
\begin{prop}
\label{poincare}
$H^*_T(F_4/T)$ is free as a $\Z$-module and its Poincar\'e series is
\[
P(H^*(ET\times_T F_4/T),x)= \frac{1}{(1-x^2)^4}(1+x^8+x^{16}) \prod_{i=1}^4 \frac{1-x^{4i}}{1-x^2}.
\]
\end{prop}

By the Serre spectral sequence for the fibration sequence \eqref{fib. seq.},
we see that generators of $H^*_T(F_4/T)$ come from the cohomology of $F_4/T$ or $BT$.
Let us define the corresponding GKM functions $t_i$, $\gamma$, $\tau_i$, $\gamma_1$ and $\gamma_3 \in {\rm Map}(W(F_4), H^*(BT))$
for $1 \leq i \leq 4$ and GKM functions $\gamma_2$ and $\gamma_4$ to state our results simpler.
For any $w \in W(F_4)$
\begin{align*}
   t_i(w) &= t_i \:\:(i=1,\ldots,4)\\
\gamma(w) &= \gamma \\
\tau_i(w) &= w(t_i) \:\:(i=1,\ldots,4)\\
 \gamma_j &= \frac{1}{2} (e_j(\tau)-e_j(t))\:\:(j=1,2,3),
\end{align*}
and
\begin{equation*}
\gamma_4(w)=
\begin{cases}
0			&w \in W \sqcup \rho^2 \kappa W,\\
e_4(\rho^2t)	&w \in \rho^2 W \sqcup \rho \kappa W,\\
-e_4(t)		&w \in \rho W \sqcup \kappa W.
\end{cases}
\end{equation*}
Moreover we define $\omega=e_4(\tau) -e_4(t) -2\gamma_4$.
Then
\begin{equation}
\label{omega(w)}
\omega(w)=
\begin{cases}
0			&w \in W \sqcup \kappa W,\\
-e_4(\rho^2t)	&w \in \rho W \sqcup \rho \kappa W,\\
e_4(\rho t)		&w \in \rho^2 W \sqcup \rho^2 \kappa W.
\end{cases}
\end{equation}

Since $t_i$'s and $\gamma$ are constant functions, they are GKM functions.
A straightforward calculation shows that the following relation holds.
\begin{equation}
\label{id}
e_4(t) +e_4(\rho t) +e_4(\rho^2 t) =0
\end{equation}
By the schematic diagram \eqref{sch. diag. F_4} of $\mathcal{F}_4$,
one can see that $\gamma_4$ is a GKM function
since $e_4(\re t)$ is the product of all $\re t_1$, $\re t_2$, $\re t_3$, and $\re t_4$ for $\ep=0,1,2$.
The following calculation shows $\tau_i$'s satisfy the GKM condition.
For any edge $ww'$ which satisfies $w'=w\sigma_\alpha$, we have
\begin{align*}
\tau_i(w) -\tau_i(w')&=w(t_i) -w'(t_i)\\
&=w\biggl(t_i -\Big(t_i -2\frac{(t_i,\alpha)}{(\alpha,\alpha)}\alpha\Big)\biggr)\\
&=2\frac{(t_i,\alpha)}{(\alpha,\alpha)}w\alpha.
\end{align*}
Since GKM functions form a ring, for $j=1,2,3$, we see that $\gamma_j$'s are functions from $W(F_4)$ to $H^*(BT)\otimes \Z[\frac{1}{2}]$
which satisfy the GKM condition tensoring with $\Z[\frac{1}{2}]$.
The following calculations show that $\gamma_j$'s are actually $H^*(BT)$-valued functions.
Let us extend $\rho$ to an automorphism of $H^*(BT)$ naturally.
For $w \in W \sqcup \kappa W = W(Spin(9))$ and $\ep =0,1,2$,
\begin{align*}
\gamma_j(\re w)	&= \frac{1}{2}(e_j(\tau) -e_j(t))(\re w)\\
			&= \frac{1}{2}(\re e_j(w(t)) -e_j(t))\\
			&= \re \Bigl(\frac{1}{2}(e_j(w(t)) -e_j(t))\Bigr) +\frac{1}{2}(e_j(\re t) -e_j(t)).
\end{align*}
Since $w$ only permutes $t_i$'s and changes their signs, it is obvious that $\frac{1}{2}(e_j(w(t)) -e_j(t)) \in H^*(BT)$.
Then $\re (\frac{1}{2}(e_j(w(t)) -e_j(t))) \in H^*(BT)$.
On the other hand one can see that $\frac{1}{2}(e_j(\re t) -e_j(t)) \in H^*(BT)$ for $\ep =0,1,2$ as follows.
When $\ep=0$, $\frac{1}{2}(e_j(\re t) -e_j(t)) = 0$ and it is contained in $H^*(BT)$.
When $\ep=1,2$, Table \ref{(e_j(re t)-e_j(t))/2} shows the value of $\frac{1}{2}(e_j(\re t) -e_j(t))$ for $j=1,2,3$.
Then $\gamma_j$ is a $H^*(BT)$-valued function and then a GKM function.
\begin{table}[h]
\tablinesep =15pt
  \caption{the value of $\frac{1}{2}(e_j(\re t) -e_j(t))$}
  \centering
  \begin{tabular}{c|ccc|}
  \label{(e_j(re t)-e_j(t))/2}
		& $j=1$ 		& $j=2$ 			& $j=3$ \\ \hline
    $\ep=1$ & $-\gamma -t_4$ 	& $-\gamma^2 +t_4^2$ 	& $t_4\gamma(\gamma-t_4)-t_4(t_1t_2+t_2t_3+t_3t_1)$  \\
    $\ep=2$	& $-2\gamma +t_4$	& $(-2\gamma +t_4)t_4$ 	& $\gamma^3-t_4\gamma^2-\gamma(t_1t_2+t_2t_3+t_3t_1)$ \\ \hline
  \end{tabular}
\end{table}\\
The following lemma will be proved in Section \ref{pf of Lemma generate}.
\begin{lem}[see {\cite[Lemma 5.4]{FIM}}]
\label{generate}
Let $\F_4$ be the GKM graph of $F_4/T$, then
$H^*(\F_4)$ is generated by the GKM functions $t_i$, $\gamma$, $\tau_i$, $\gamma_i$, $\omega$ $(i=1,2,3,4)$ as a ring.
\end{lem}

By the fibration sequence \eqref{fib. seq.},
we can expect some relations hold in $H^*(\F_4)$, which come from the relations of $H^*(BT)$ and $H^*(F_4/T)$.
Proposition \ref{relations} claims the corresponding relations hold in $H^*(\F_4)$.
\begin{prop}
\label{relations}
The following relations hold in $H^*(\F_4) \subset {\rm Map}(W(F_4),H^*(BT))$;
\begin{align}
& r_1' = e_1(t) -2\gamma = 0, \label{gamma}\\
& R_1 = e_1(\tau)-e_1(t) -2\gamma_1 = 0,\label{R_1}		\\
& R_2 = e_2(\tau)-e_2(t) -2\gamma_2 = 0,\label{R_2}		\\
& R_3 = e_3(\tau)-e_3(t) -2\gamma_3 = 0,\label{R_3}		\\
& R_4 = e_4(\tau)-e_4(t) -2\gamma_4 -\omega = 0, \label{R_4}	\\
& r_2 = \sum_{j=1}^{2} (-1)^j \gamma_j(\gamma_{2-j} +e_{2-j}(t)) = 0, \label{r_2}\\
& r_4 = \sum_{j=1}^{4} (-1)^j \gamma_j(\gamma_{4-j} +e_{4-j}(t)) -\omega = 0,\label{r_4}\\
& r_6 = \sum_{j=2}^{4} (-1)^j \gamma_j(\gamma_{6-j} +e_{6-j}(t)) +(\gamma_2 + \gamma^2) \omega = 0,\label{r_6}\\
& r_8 = \gamma_4(\gamma_{4} +e_{4}(t)) +\omega^2 + (\gamma_4 -e_4(\rho t))\omega = 0, \label{r_8}\\
& r_{12} = \omega(\omega-e_4(\rho t))(\omega+e_4(\rho^2 t))=0. \label{omega}
\end{align}
\end{prop}
Proposition \ref{relations} is proved in Section \ref{pf of Prop relations}.
The following lemma is proved in Section \ref{pf of Lemma free}.
\begin{lem}
\label{free}
$\Z[t_i, \gamma, \tau_i, \gamma_i, \omega \suchthat 1 \leq i \leq 4]/(r_1', R_i, r_{2i}, r_{12} \suchthat 1 \leq i \leq 4)$
is free as a $\Z$-module,
and its Poincar\'e series coincides with that of $H^*_T(F_4/T)$.
\end{lem}

Now we are ready to prove Theorem \ref{main}.
\begin{proof}[Proof of Theorem \ref{main}]
Let $I$ denote the ideal $(r_1', R_i, r_{2i}, r_{12} \suchthat 1 \leq i \leq 4)$
in the polynomial ring $\Z[t_i, \gamma, \tau_i, \gamma_i, \omega \suchthat1 \leq i \leq 4]$.
We have a surjective ring homomorphism
\[
\Z[t_i, \gamma, \tau_i, \gamma_i, \omega \suchthat1 \leq i \leq 4] \to H^*(\F_4)
\]
by Lemma \ref{generate},
and it factors through $\Z[t_i, \gamma, \tau_i, \gamma_i, \omega \suchthat1 \leq i \leq 4]/I \to H^*(\F_4)$ by Proposition \ref{relations}.
It follows from Proposition \ref{poincare} and Lemma \ref{free} that
$H^*(\F_4)$ and $\Z[t_i, \gamma, \tau_i, \gamma_i, \omega \suchthat1 \leq i \leq 4]/I$ are free as $\Z$-modules.
Moreover Lemma \ref{free} claims that $\Z[t_i, \gamma, \tau_i, \gamma_i, \omega \suchthat1 \leq i \leq 4]/I$ and $H^*(\F_4)$ have the same rank in each degree.
Therefore the ring homomorphism $\Z[t_i, \gamma, \tau_i, \gamma_i, \omega \suchthat1 \leq i \leq 4]/I \to H^*(\F_4)$ is an isomorphism
and Theorem \ref{main} is proved by Theorem \ref{HHH}.
\end{proof}

\section{Proof of Lemma \ref{generate}}\label{pf of Lemma generate}
First we introduce some notation for the proof of Lemma \ref{generate}.
For a positive integer $n$, let $[n]$ and $\pm[n]$ be $\{i \in \Z \suchthat 1\leq i \leq n \}$ and $\{\pm i \in \Z \suchthat 1\leq i \leq n \}$, respectively.
For $1 \leq n \leq 4$, let $I_n$ denote an ordered $n$-tuple $(i_1, \ldots, i_n)$ of elements of $[4]$ which does not include the same entries,
and $I'_n$ denote an ordered $n$-tuple $(i'_1, \ldots, i'_n)$ of elements of $\pm[4]$ such that $|i'_k| \neq |i'_l|$ for $k\neq l$.
We often regard $I_n$, $I'_n$ as the $n$-subsets of $[4]$ by the following maps.
\[
(i_1, \ldots, i_n) \mapsto \{i_1, \ldots, i_n\}, \quad (i_1', \ldots, i_n') \mapsto \{|i_1'|, \ldots, |i_n'|\}
\]
Let $t_{i'} = {\rm sgn}(i') t_{|i'|}$. For $\ep = 0,1,2$, we define a subset $\re W\ii$ of $W(F_4)$ as:
\[
\re W\ii = \{ w \in W(F_4) \suchthat w \in \re W(Spin(9)),\ w(t_{i_k})= \re t_{i'_k} \ (1\leq k \leq n)\}
\]
We define $I_0$ and $I'_0$ to be the empty set.
Note that $\re W\ii[n-1]$ includes $\re W\ii$ and decomposes as follows.
\begin{equation}
\label{rem_decomp}
\re W\ii[n-1] = \coprod_{i_n \in [4] \sm I_{n-1}} \re W\iii{i_n}{i'_n} \ \sqcup \coprod_{i_n \in [4] \sm I_{n-1}} \re W\iii{i_n}{-i'_n}
\end{equation}

For a set $S =\{j_1, \ldots, j_k\}$ of natural numbers with $j_1 < \cdots < j_k$,
let $x_{S}$ denote a sequence $(x_{j_1}, \ldots, x_{j_k})$ for $x=t$, $\rho t$, $\rho^2 t$, $\tau$.
For $n\geq 0$, $j\leq 4$, and $\ep=0,1,2$, let $\gamma_j^{(\ep)}\ii$ be a function from $\rho^\ep W\ii$ to $\Z[\frac{1}{2}][t_1, t_2, t_3, t_4]$
defined as
\[
\gamma_j^{(\ep)}\ii = \frac{1}{2}(e_j(\tau_{[4] \sm I_n}) -e_j(\re t_{[4] \sm I'_n})),
\]
where $I_n$ and $I'_n$ in the right hand side are regarded as subsets of $[4]$.
When $n=0$ we abbreviate $\gamma^{(\ep)}_{j}{}^{\emptyset}_{\emptyset}$ by $\gamma^{(\ep)}_{j}$.
If $j \leq 0$ or $j >4-n$, we define $\gamma_j^{(\ep)}\ii = 0$.

We define a function $f^{(\ep)}{}^{I_{n-1}}_{i'_n}$ which is useful in the proof of Lemma \ref{generate} as:
\[
f^{(\ep)}{}^{I_{n-1}}_{i'_n}=\frac{1}{2}\prod_{k \in [4]\sm I_{n-1}}(\tau_{k} - \re t_{i'_n}).
\]
This function is $H^*(BT)$-valued on $\re W\ii[n-1]$,
since for any $w \in \re W\ii[n-1]$ there exists $k \in [4] \sm I_{n-1}$ such that $w \in \re W\iii{k}{i'_n} \sqcup \re W\iii{k}{-i'_n}$
by the decomposition \eqref{rem_decomp},
and then $w(t_k) -\re t_{i'_n}$ equals to $0$ or $-2 \re t_{i'_n}$.
Especially  we have
\begin{equation}
\label{f^(ep)}
f^{(\ep)}{}^{I_{n-1}}_{i'_n} (w) =
 \begin{cases}
  0 & w \in \coprod_{k \in [4]\sm I_{n-1}}\re W\iii{k}{i'_n},\\
  -\re t_{i'_n} \prod_{k \in [4]\sm I'_n} (\re t_k -\re t_{i'_n}) & w \in \coprod_{k \in [4]\sm I_{n-1}}\re W\iii{k}{-i'_n}.
 \end{cases}
\end{equation}
Let $R$ denote the subring of $H^*(\F_4)$ generated by $t_i$'s, $\gamma$, $\tau_i$'s, and $\gamma_i$'s ($1 \leq i \leq 4$).
The following proposition claims that this function $f^{(\ep)}{}^{I_{n-1}}_{i'_n}$ can be replaced partly
by an element of $R$.
\begin{prop}
\label{1/2}
For $1 \leq n \leq 4$, there is a polynomial in $\gamma_1$, $\gamma_2$, $\gamma_3$, $\gamma_4$ over $H^*(BT)$,
which coincides with the function $f^{(\ep)}{}^{I_{n-1}}_{i'_n}$ on $\re W\ii[n-1]$.
\end{prop}
Proposition \ref{1/2} is a consequence of Lemma \ref{rep} and \ref{f(n)} below.
\begin{lem}
\label{rep}
For $1 \leq n \leq 4$, there is a polynomial in $\gamma_j^{(\ep)}\ii[n-1]$'s $(1 \leq j \leq 4-(n-1))$ over $H^*(BT)$,
which coincides with $f^{(\ep)}{}^{I_{n-1}}_{i'_n}$ on $\re W\ii[n-1]$.
\end{lem}
\begin{lem}[cf. {\cite[Lemma 5.3]{FIM}}]
\label{f(n)}
For $1 \leq n \leq 4$ and $1 \leq j \leq 4-n$, there is a polynomial in $\gamma^{(\ep)}_{1}\ii[n-1], \ldots, \gamma^{(\ep)}_{4-n}\ii[n-1]$
over $H^*(BT)$,
which coincides with $\gamma_j^{(\ep)}\iii{i_n}{i'_n}$ on $\re W\iii{i_n}{i'_n}$.
More explicitly,
\[
\gamma_j^{(\ep)}\ii =
\begin{cases}
\sum\limits_{k=0}^{j-1} \gamma^{(\ep)}_{j-k}\ii[n-1](-\rho^\ep t_{i'_n})^k & {\rm sgn}\ i'_n = 1\\
\sum\limits_{k=0}^{j-1} \gamma^{(\ep)}_{j-k}\ii[n-1](-\rho^\ep t_{i'_n})^k
 + \sum\limits_{k=1}^j e_{j-k}(\re t_{[4]\sm I'_n})(-\re t_{i'_n})^k & {\rm sgn}\ i'_n = -1.
\end{cases}
\]
\end{lem}

\begin{proof}[Proof of Proposition \ref{1/2}]
By Lemma \ref{rep},
there is a polynomial in $\gamma_j^{(\ep)}\ii[n-1]$'s $(1 \leq j \leq 4-(n-1))$ over $H^*(BT)$,
which coincides with $f^{(\ep)}{}^{I_{n-1}}_{i'_n}$ on $\re W\ii[n-1]$ for $\ep=0,1,2$.
Then by Lemma \ref{f(n)} $\gamma_j^{(\ep)}\iii{i_n}{i'_n}$ can be replaced by some polynomial
in $\gamma^{(\ep)}_{1}\ii[n-1], \ldots, \gamma^{(\ep)}_{4-n}\ii[n-1]$ over $H^*(BT)$.
By a descending induction on $n$ we reached to a polynomial in $\gamma_1^{(\ep)}$, $\gamma_2^{(\ep)}$, $\gamma_3^{(\ep)}$, $\gamma_4^{(\ep)}$ over $H^*(BT)$,
which coincides with $f^{(\ep)}{}^{I_{n-1}}_{i'_n}$ on $\re W\ii[n-1]$ for $\ep=0,1,2$.
Next we need to show that
$\gamma_j - \gamma^{(\ep)}_{j} \in H^*(BT)$ on $\re W(Spin(9))$ for $1 \leq j \leq 4$ and $\ep=0,1,2$
to complete the proof of Proposition \ref{1/2}.
By definition we have
\[
\gamma^{(\ep)}_{j} = \gamma_j +\frac{1}{2}(e_j(t)-e_j(\re t)) \quad (j=1,2,3).
\]
For $\ep=0,1,2$ and $j=1,2,3$, Table \ref{(e_j(re t)-e_j(t))/2} shows that $(e_j(t)-e_j(\re t))/2 \in H^*(BT)$
and then $\gamma_j - \gamma_j^{(\ep)} \in H^*(BT)$ on $\re W(Spin(9))$.
By the definition of $\gamma_4$ and the equation \eqref{id}, we have
\begin{align*}
\gamma^{(0)}_{4} &= \gamma_4				&&\text{on } W(Spin(9)),\\
\gamma^{(1)}_{4} &= \gamma_4 +e_4(t)		&&\text{on } \rho W(Spin(9)),\\
\gamma^{(2)}_{4} &= \gamma_4 -e_4(\rho^2 t)	&&\text{on } \rho^2 W(Spin(9)).
\end{align*}
Therefore there is a polynomial in $\gamma_1$, $\gamma_2$, $\gamma_3$, $\gamma_4$ over $H^*(BT)$,
which coincides with the function $f^{(\ep)}{}^{I_{n-1}}_{i'_n}$ on $\re W\ii[n-1]$.
\end{proof}

\begin{proof}[Proof of Lemma \ref{rep}]
Without loss of generality,
we may suppose that $I_{n-1} = (1, \ldots, n-1)$.
Note that $e_j(x_{S}) = 0$ for $j > \# S$ or $j < 0$,
and that we have
\begin{equation}
\label{e_j}
e_j(x_1, \ldots, x_{m-1}, x_m) = e_j(x_1, \ldots, x_{m-1}) +e_{j-1}(x_1, \ldots, x_{m-1})x_m.
\end{equation}
By the definition of $\gamma_j^{(\ep)}\ii[n-1]$ we can expand the GKM function $f^{(\ep)}{}^{I_{n-1}}_{i'_n}$ as follows.
\begin{align*}
\frac{1}{2}\prod_{l=0}^{4-n}(\tau_{n+l} - \re t_{i'_n}) =& \frac{1}{2} \sum_{j=0}^{5-n} e_j(\tau_{[4] \sm I_{n-1}})(-\re t_{i'_n})^{5-n-j} \nonumber\\
=& \frac{1}{2} \sum_{j=0}^{5-n} (2\gamma_j^{(\ep)}\ii[n-1] + e_j(\re t_{[4] \sm I'_{n-1}}))(-\re t_{i'_n})^{5-n-j}
\end{align*}
Pay attention to the sign of $i'_n$ and recall that $[4]\sm I'_{n-1} = \{i \in [4] \suchthat \pm i \not \in I'_{n-1} \}$.
By \eqref{e_j}, the above equals to
\begin{align*}
\sum_{j=0}^{5-n} \gamma_j^{(\ep)}\ii[n-1](-\re t_{i'_n})^{5-n-j} + \frac{1}{2}\sum_{j=0}^{5-n} (e_j(\re t_{[4] \sm I'_{n}}) + e_{j-1}(\re t_{[4] \sm I'_{n}}) \re t_{|i'_n|})(-\re t_{i'_n})^{5-n-j}\\
=\begin{cases}
\sum\limits_{j=0}^{5-n} \gamma_j^{(\ep)}\ii[n-1](-\re t_{i'_n})^{5-n-j}		& {\rm sgn}\ i_n' =1\\
\sum\limits_{j=0}^{5-n} \gamma_j^{(\ep)}\ii[n-1](-\re t_{i'_n})^{5-n-j}
+\sum\limits_{j=0}^{4-n}e_j(\re t_{[4] \sm I'_{n}})(-\re t_{i'_n})^{5-n-j} 	& {\rm sgn}\ i_n' =-1.
\end{cases}
\end{align*}
\end{proof}

\begin{proof}[Proof of Lemma \ref{f(n)}]
The relation $\tau_{i_n} = \re t_{i'_n}$ holds on $\re W\iii{i_n}{i'_n}$.
Then we have
\begin{align*}
 &\ \gamma_j^{(\ep)}\ii[n-1] - \gamma_j^{(\ep)}\iii{i_n}{i'_n} \\
=&\ \frac{1}{2}(e_j(\tau_{i \in [4] \sm I_{n-1}}) - e_j(\re t_{i' \in [4]\sm I'_{n-1}})) - \frac{1}{2}(e_j(\tau_{i \in [4] \sm I_{n}}) - e_j(\re t_{i' \in [4] \sm I'_{n}}))\\
=&\ \frac{1}{2}(e_{j-1}(\tau_{i \in [4] \sm I_{n}})\tau_{i_n} - e_{j-1}(\re t_{i' \in [4] \sm I'_{n}})\re t_{|i'_n|})\\
=&\ 
\begin{cases}
\gamma_{j-1}^{(\ep)}\ii \re t_{i'_n} & {\rm sgn}\ i'_n = 1\\
\gamma_{j-1}^{(\ep)}\ii \re t_{i'_n} + e_{j-1}(\re t_{i' \in [4] \sm I'_{n}})\re t_{i'_n} & {\rm sgn}\ i'_n = -1.
\end{cases}
\end{align*}
Iterated use of this equation shows that
\[
\gamma_j^{(\ep)}\iii{i_n}{i'_n} =
\begin{cases}
\sum\limits_{k=0}^{j-1} \gamma^{(\ep)}_{j-k}\ii[n-1](-\rho^\ep t_{i'_n})^k & {\rm sgn}\ i'_n = 1\\
\sum\limits_{k=0}^{j-1} \gamma^{(\ep)}_{j-k}\ii[n-1](-\rho^\ep t_{i'_n})^k
 + \sum\limits_{k=1}^j e_{j-k}(\re t_{[4]\sm I'_n})(-\re t_{i'_n})^k & {\rm sgn}\ i'_n = -1.
\end{cases}
\]

\end{proof}

Now we are ready to prove Lemma \ref{generate}.
\begin{proof}[Proof of Lemma \ref{generate}]
We show that any GKM function $h \in H^*(\F_4)$
belongs to the subring $R$ generated by $t_i$'s, $\gamma$, $\tau_i$'s, $\gamma_i$'s, and $\omega$ $(1 \leq i \leq 4)$.
By the definition of $\rho$,
the set of all vertices $W(F_4)$ of $\F_4$ decomposes as:
\[
W(F_4) = W(Spin(9)) \sqcup \rho W(Spin(9)) \sqcup \rho^2 W(Spin(9)).
\]
For each $\ep=0,1,2$, $\re W(Spin(9))$ has a filtration
\[
\re W\ii[4] \subset \cdots \subset \re W\ii[n] \subset \re W\ii[n-1] \subset \cdots \subset \re W\ii[0] = \re W(Spin(9)).
\]
By descending induction on $n$,
we will show that any GKM function $h$ can be modified to be $0$ on $\re W\ii$
by subtracting some GKM function in $R$.
Moreover, in the induction step on $n$,
we give an induction to fill the decomposition \eqref{rem_decomp} of $\re W\ii[n-1]$.

Let $0 \leq n \leq 4$.
The following claim in the case where $n=0$ shows that $h$ can be modified to be $0$ on $W(Spin(9))$.
\begin{claim}[$n$]
\label{claim1}
For any ordered $n$-tuples $I_n$, $I'_n$ and any function $h$
from $W\ii$ to $H^*(BT)$ which satisfies the GKM condition on $W\ii$,
there is a GKM function $G \in R$
which coincides with $h$ on $W\ii$.
\end{claim}
We show this claim by descending induction on $n$.
For $n=4$, since $W\ii[4]$ is a one point set, the claim holds obviously.
Assume Claim \ref{claim1} ($n$) holds, and fix $I_n = (i_1, \ldots, i_n)$ and $I'_n  = (i_1', \ldots, i_n')$.
Then we have a GKM function which coincides with $h$ on $W\ii$.
Subtracting this GKM function from $h$, we may assume $h$ vanishes on $W\ii$.
We give an induction to fill the decomposition \eqref{rem_decomp} of $W\ii[n-1]$ as follows.
For any $k \in [4] \sm I_n$, let $\sigma_k$ denote the reflection associated with $t_k -t_{i_n}$.
Then $\sigma_k$ interchanges $t_k$ and $t_{i_n}$, and for any $w \in W\iii{k}{i'_n}$, $w \sigma_k$ is contained in $W\iii{i_n}{i'_n}$.
By the GKM condition,
$h(w) - h(w\sigma_k) = h(w)$ belongs to the ideal generated by $w(t_{i_n} - t_k) = w(t_{i_n}) - t_{i'_n} = \tau_{i_n}(w) -t_{i'_n}$.
Put $k_0, \ldots, k_{4-n} \in [4]\setminus I_{n-1}$ as $k_0 = i_n$, $k_s < k_t$ for $1\leq s <t$
and $\{ k_0, \ldots, k_{4-n}\} \cup I_{n-1} = [4]$.
\begin{claim}[$t$]
\label{claimt}
If $h$ is a GKM function which vanishes on $\coprod_{s < t} W\iii{k_s}{i'_n}$,
there is a GKM function $G \in R$
such that $h$ coincides with $\prod_{s < t}(\tau_{k_s} - t_{i'_n})G$ on $W\iii{k_t}{i'_n}$.
\end{claim}
We show this claim by induction on $t$ ($0 \leq t \leq 4-n$).
Without loss of generality,
we may suppose that $I_{n-1} = (1, \ldots, n-1)$ and $k_0 = i_n = n$, $k_1 = n+1, \ldots, k_{4-n} = 4$.
We rephrase Claim \ref{claimt} ($t$) as:
\addtocounter{claim}{-1}
\begin{claim}[$k$]
If $h$ vanishes on $\coprod_{0 \leq l < k} W\iii{n+l}{i'_n}$,
there is a GKM function $G \in R$
such that $h$ coincides with $\prod_{0\leq l < k}(\tau_{n+l} - t_{i'_n})G$ on $W\iii{n+k}{i'_n}$.
\end{claim}
Obviously $\prod_{0\leq l < k}(\tau_{n+l} - t_{i'_n})$ vanishes on $\coprod_{0\leq l < k} W\iii{n+l}{i'_n}$.
For $w \in W\iii{n+k}{i'_n}$, by the GKM condition, there is an element $g_w \in H^*(BT)$ such that
\[
h(w) = \Bigl(\prod_{0 \leq l < k}(\tau_{n+l} - t_{i'_n})(w)\Bigr)g_w.
\]
One can verify that a function $G' \colon W\iii{n+k}{i'_n} \to H^*(BT)$ given by
\[
G'(w)=g_w
\]
satisfies the GKM condition on $W\iii{n+k}{i'_n}$ as follows.
Assume that two vertices $w$, $w' \in W\iii{n+k}{i'_n}$ of $\F_4$ satisfy $w' = w\sigma_\alpha$ for some positive root $\alpha$.
Then $\alpha = t_i -t_j$, where $i < j$ and $i$, $j \in \{ m \in \Z \suchthat n \leq m \leq n+k-1 \text{ or } n+k+1 \leq m \leq 4\}$.
When $i < j < n+k$ or $n+k <i < j$, the GKM condition says
\begin{align*}
 h(w) -h(w')
=&\ \Bigl(\prod_{0\leq l < k} (w(t_{n+l}) - t_{i'_n})\Bigr)G'(w) - \Bigl(\prod_{0\leq l < k} (w\sigma_{t_i-t_j}(t_{n+l}) - t_{i'_n})\Bigr)G'(w')\\
=&\ \Bigl(\prod_{0\leq l < k} (w(t_{n+l}) - t_{i'_n})\Bigr)(G'(w)-G'(w'))
\end{align*}
belongs to the ideal $(w(t_i-t_j))$.
Since $w(t_{n+l})-t_{i'_n}$ and $w(t_i -t_j)$ are relatively prime,
$G'(w)-G'(w')$ also belongs to the ideal $(w(t_i-t_j))$.
When $i < n+k < j$, the GKM condition says
\begin{align*}
 &\ h(w) -h(w')\\
=&\ \Bigl(\prod_{0\leq l < k} (w(t_{n+l}) - t_{i'_n})\Bigr)G'(w) - \Bigl(\prod_{0\leq l < k} (w\sigma_{t_i-t_j}(t_{n+l}) - t_{i'_n})\Bigr)G'(w')\\
=&\ \Bigl(\prod_{0\leq l < k,\ l\neq i} (w(t_{n+l}) - t_{i'_n})\Bigr)\Bigl((w(t_{i})-t_{i'_n})(G'(w)-G'(w')) +(w(t_{i})-w(t_{j}))G'(w')\Bigr)
\end{align*}
belongs to the ideal $(w(t_i-t_j))$.
Since $w(t_{n+l})-t_{i'_n}$ and $w(t_i -t_j)$ are relatively prime,
$G'(w)-G'(w')$ also belongs to the ideal $(w(t_i-t_j))$.
Hence the function $G'$ satisfies the GKM condition on $W\iii{n+k}{i'_n}$.

By (descending) induction on $n$ there is a GKM function $G \in R$
such that $G$ and $G'$ coincide on $W\iii{n+k}{i'_n}$.
Then 
\[
h - \Bigl(\prod_{0 \leq l < k}(\tau_{n+l} - t_{i'_n})\Bigr)G = 0 \quad {\rm on} \coprod_{0 \leq l \leq k} W\iii{n+l}{i'_n}.
\]
Therefore the induction on $k$ proceeds.

Next we fill the other half of the decomposition \eqref{rem_decomp}.
Note that when $I_{n-1}=(1, \ldots, n-1)$,
\[
f^{(0)}{}^{I_{n-1}}_{i'_n} = \frac{1}{2}\prod_{0\leq l \leq 4-n}(\tau_{n+l} - t_{i'_n}).
\]
Let $0 \leq k' \leq 4-n$.
\begin{claim}[$k'$]
If $h$ vanishes on $\coprod_{0 \leq l \leq 4-n} W\iii{n+l}{i'_n} \sqcup \coprod_{0\leq l < k'}W\iii{n+l}{-i'_n}$.
There is a GKM function $G \in R$
such that $h$ coincides with
$f^{(0)}{}^{I_{n-1}}_{i'_n}\prod_{0\leq l < k'}(\tau_{n+l} + t_{i'_n})G$ on $W\iii{n+k'}{-i'_n}$.
\end{claim}
We show this claim by induction on $k'$.
For $w \in W\iii{n+k'}{-i'_n}$, by the GKM condition, $h(w)$ belongs to the ideal generated by the following elements of $H^*(\F_4)$.
\begin{align*}
w(t_{n+l}-t_{n+k'})&=w(t_{n+l})+t_{i'_n}&& \text{ for } 0 \leq l < k'\\
w(t_{n+l}+t_{n+k'})&=w(t_{n+l})-t_{i'_n}&& \text{ for } 0 \leq l \leq 4-n,\ l \neq k'\\
w(t_{n+k'}) &= -t_{i'_n}&&
\end{align*}
For $w \in W\iii{n+k'}{-i'_n}$, by the equation \eqref{f^(ep)}, there is an element $g_w \in H^*(BT)$ such that
\[
h(w) = f^{(0)}{}^{I_{n-1}}_{i'_n}(w)\Bigl(\prod_{0\leq l < k'}(\tau_{n+l} + t_{i'_n})(w)\Bigr)g_w.
\]
One can verify that a function $G'$ given by $G'(w)=g_w$ satisfies the GKM condition on $W\iii{n+k'}{-i'_n}$ similarly as above.
By (descending) induction on $n$ there is a GKM function $G \in R$
such that $G$ and $G'$ coincide on $W\iii{n+k'}{-i'_n}$.
Then 
\begin{align*}
h - f^{(0)}{}^{I_{n-1}}_{i'_n} \Bigl(\prod_{0\leq l < k'}(\tau_{n+l} + t_{i'_n})\Bigr) G = 0\\
{\rm on} \coprod_{0\leq l \leq 4-n}W\iii{n+l}{i'_n} \sqcup \coprod_{0\leq l \leq k'}W\iii{n+l}{-i'_n}.
\end{align*}
Therefore the induction on $k'$ proceeds.
By Proposition \ref{1/2} the function
\[
f^{(0)}{}^{I_{n-1}}_{i'_n}=\frac{1}{2}\prod_{0\leq l \leq 4-n} (\tau_{n+l} -t_{i'_n})
\]
can be replaced by a polynomial in $\gamma_j$'s $(1 \leq j \leq 4)$ over $H^*(BT)$.
Therefore the (descending) induction on $n$ proceeds,
and we may assume $h$ vanishes on $W(Spin(9))=W \sqcup \kappa W$.

Next we show that for a GKM function $h$ which vanishes on $W(Spin(9))$,
there is a GKM function $G \in R$
such that $h -\omega G = 0$ on $W(Spin(9)) \sqcup \rho W(Spin(9))$,
where $\omega$ vanishes on $W(Spin(9))$.
Recall that the schematic diagram \eqref{sch. diag. F_4} says that
each $w \in \rho W \sqcup \rho \kappa W$ are adjacent to $4$ vertices of $W \sqcup \kappa W$,
and the labels of these edges are $\rho^2 t_i$ $(1 \leq i \leq 4)$ and different each other.
The GKM condition says that for $w \in \rho W(Spin(9))$,
$h(w)$ belongs to the ideal $(\rho^2 t_i\suchthat i = 1,2,3,4)$.
For $w \in \rho W(Spin(9))$, there is an element $g_w \in H^*(BT)$ such that
\[
h(w) = -e_4(\rho^2 t)g_w = \omega(w)g_w.
\]
It is obvious that a function $G'$ given by $G'(w)=g_w$ satisfies the GKM condition on $\rho W(Spin(9))$,
since the edges in the GKM subgraph induced by $\rho W(Spin(9))$ have the long roots or $\rho t_i$ as their labels
and the all positive roots of $F_4$ are relatively prime in $H^*(BT)$.
Then we claim that there is a GKM function $G$ such that $G=G'$ on $\rho W(Spin(9))$.
This claim is proved similarly as above, changing $W\ii$ to $\rho W\ii$, $\tau_{k_s}-t_{i'_n}$ to $\tau_{k_s}-\rho t_{i'_n}$,
and $f^{(0)}{}^{I_{n-1}}_{i'_n}$ to $f^{(1)}{}^{I_{n-1}}_{i'_n}$.

Finally we show that for a GKM function $h$ which vanishes on $W(Spin(9))\sqcup \rho W(Spin(9))$,
there is a GKM function $G \in R$
such that $h - \omega(\omega+e_4(\rho^2 t))G = 0$ as a GKM function on the whole $W(F_4)$,
where $\omega+e_4(\rho^2 t)$ vanishes on $\rho W(Spin(9))$.
It is proved similarly as above that
for $w \in \rho^2 W(Spin(9))$, $h(w)$ belongs to the ideal $(t_i, \rho t_i \suchthat i = 1,2,3,4)$.
For $w \in \rho^2 W(Spin(9))$, there is an element $g_w \in H^*(BT)$ such that
\[
h(w) = -e_4(\rho t)e_4(t)g_w = \omega(w)(\omega(w)+e_4(\rho^2 t))g_w,
\]
where the latter equality is due to the relation \eqref{id}.
Then we claim that a function $G'$ given by $G'(w)=g_w$ satisfies the GKM condition on $\rho^2 W(Spin(9))$,
and that there is a GKM function $G$ such that $G=G'$ on $\rho^2 W(Spin(9))$.
This claim is proved similarly as above, changing $W\ii$ to $\rho^2 W\ii$, $\tau_{k_s}-t_{i'_n}$ to $\tau_{k_s}-\rho^2 t_{i'_n}$,
and $f^{(0)}{}^{I_{n-1}}_{i'_n}$ to $f^{(2)}{}^{I_{n-1}}_{i'_n}$.
The proof is completed.
\end{proof}

\section{Proof of Proposition \ref{relations}}
\label{pf of Prop relations}

We prove Proposition \ref{relations} in the similar way of \cite[Proof of Lemma 5.5]{FIM}.

\begin{proof}[Proof of Proposition \ref{relations}]
The relations \eqref{gamma}, \eqref{R_1}, \eqref{R_2}, \eqref{R_3}, and \eqref{R_4} hold obviously by definition,
and the relation \eqref{omega} holds by \eqref{omega(w)}.
To show that \eqref{r_2}, \eqref{r_4}, \eqref{r_6}, and \eqref{r_8} hold,
we claim that the following relations hold in $H^*_T(\F_4)$.
\begin{align}
&e_1(\tau^2) - e_1(t^2) = 0			\label{e1}\\
&e_2(\tau^2) - e_2(t^2) -6\omega = 0	\label{e2}\\
&e_3(\tau^2) - e_3(t^2) -e_1(t^2)\omega = 0	\label{e3}\\
&e_4(\tau^2) - e_4(t^2) +3\omega^2 -2(e_4(\rho t) -e_4(\rho^2 t))\omega = 0	\label{e4}
\end{align}
The left-hand side functions of these equations are constant on each $\re W(Spin(9))$ for $\ep = 0,1,2$.
Calculations of each values on $\re W(Spin(9))$ with \eqref{id} show that the equations \eqref{e1}, \eqref{e2}, \eqref{e3}, and \eqref{e4} hold.

We show that \eqref{e1}, \eqref{e2}, \eqref{e3}, and \eqref{e4} are divisible by $4$
to deduce \eqref{r_2}, \eqref{r_4}, \eqref{r_6}, and \eqref{r_8}.
Let $x$ be an indeterminate and
put $X = -6\omega x^4 +e_1(t^2)\omega x^6 +(3\omega^2 -2(e_4(\rho t) -e_4(\rho^2 t))\omega)x^8$.
It follows from \eqref{e1}, \eqref{e2}, \eqref{e3}, and \eqref{e4} that
\begin{align*}
0&=\prod_{i=1}^4 (1-\tau_i^2 x^2) - \prod_{i=1}^4 (1-t_i^2 x^2) +X \\
 &=\sum_{k=0}^4(1+(-1)^k e_k(\tau)x^k)\sum_{k=0}^4(1+e_k(\tau)x^k) - \sum_{k=0}^4(1+(-1)^k e_k(t)x^k)\sum_{k=0}^4(1+e_k(t)x^k)+X.
\end{align*}
We can erase $e_k(\tau)$ by the relation \eqref{R_1}, \eqref{R_2}, \eqref{R_3}, and \eqref{R_4},
and obtain
\begin{align*}
 & 4\sum_{k=1}^3 (-1)^k \gamma_k^2 x^{2k} -8\gamma_1 \gamma_3 x^4 +4\sum_{k=1}^3 \sum_{i=n_k}^{m_k} (-1)^i \gamma_i e_{2k-i}(t) x^{2k}\\
 & +2(2\gamma_4 +\omega)x^4 +4\gamma_2(e_4(t) +2\gamma_4 +\omega)x^6 +2e_2(t)(2\gamma_4 +\omega)x^6 +(2e_4(t) +2\gamma_4 +\omega)(2\gamma_4 +\omega)x^8 +X,
\end{align*}
where $n_k= \max\{ 1, 2k-3 \}$ and $m_k = \min\{ 3, 2k \}$.
This calculation is similar to the calculation in \cite[Proof of Lemma 5.5]{FIM},
but note that $\gamma_4 \neq \frac{1}{2}(e_4(\tau)-e_4(t))$.
Then comparing the coefficients, we obtain
\begin{align*}
0 &= -4\gamma_1^2 +4(-\gamma_1e_1(t) +\gamma_2) = 4 \sum_{j=1}^{2} (-1)^j \gamma_j(\gamma_{2-j} +e_{2-j}(t)),\\
0 &= 4\gamma_2^2 -8\gamma_1\gamma_3 +4(-\gamma_1e_3(t) +\gamma_2e_2(t) -\gamma_3e_1(t)) +4\gamma_4 -4\omega \\
  &= 4\biggr(\sum_{j=1}^{4} (-1)^j \gamma_j(\gamma_{4-j} +e_{4-j}(t)) -\omega \biggl),\\
0 &= -4\gamma_3^2 -4\gamma_3e_3(t) +4\gamma_2(e_4(t) +2\gamma_4 +\omega) +2e_2(t)(2\gamma_4 + \omega) +(e_1(t)^2 -2e_2(t))\omega \\
  &= 4\biggr( \sum_{j=2}^{4} (-1)^j \gamma_j(\gamma_{6-j} +e_{6-j}(t)) +(\gamma_2 + \gamma^2) \omega \biggl), \\
0 &= (2e_4(t) +2\gamma_4 +\omega)(2\gamma_4 +\omega) +(3\omega^2 -2(e_4(\rho t) -e_4(\rho^2 t))\omega) \\
  &= 4\Bigr( \gamma_4(\gamma_{4} +e_{4}(t)) +\omega^2 + (\gamma_4 -e_4(\rho t))\omega \Bigl). \\
\end{align*}
Regarding GKM functions as elements of ${\rm Map}(W(G),H^*(BT)\otimes \Q)$,
we can divide them by $4$ to obtain
\begin{align*}
& 0 = \sum_{j=1}^{2} (-1)^j \gamma_j(\gamma_{2-j} +e_{2-j}(t)), 						&& 0 = \sum_{j=1}^{4} (-1)^j \gamma_j(\gamma_{4-j} +e_{4-j}(t)) -\omega,\\
& 0 = \sum_{j=2}^{4} (-1)^j \gamma_j(\gamma_{6-j} +e_{6-j}(t)) +(\gamma_2 + \gamma^2) \omega,	&& 0 = \gamma_4(\gamma_{4} +e_{4}(t)) +\omega^2 + (\gamma_4 -e_4(\rho t))\omega.
\end{align*}
Since the right-hand sides of these equations remain to be polynomials in $H^*(BT)$-valued GKM functions over $\Z$,
these equations hold in $H^*(\F_4) \subset {\rm Map}(W(G),H^*(BT))$.
\end{proof}

\section{Proof of Lemma \ref{free}}
\label{pf of Lemma free}

We will prove Lemma \ref{free} by the argument of regular sequences.

\begin{dfn}
A sequence $a_1, \ldots a_n$ of elements of a ring $R$ is called regular
if, for any $i$, $a_i$ is not a zero divisor in $R/(a_1,\ldots,a_{i-1})$.
\end{dfn}

The following theorems and propositions are useful.
Proposition \ref{a+b} and \ref{ai} are obvious by definition.
\begin{prop}
\label{a+b}
If $a_1, \ldots, a_n$ is a regular sequence,
then so is $a_1, \ldots, a_{i-1}, a_i + b, a_{i+1}, \ldots, a_n$ for $1 \leq i \leq n$ and any $b \in ( a_1, \ldots, a_{i-1} )$.
\end{prop}
\begin{prop}
\label{ai}
If $a_1, \ldots, a_n$ is a regular sequence,
then so is $a_1, \ldots, a_{i-1}, a_{i+1}, \ldots, a_n$ for $1 \leq i \leq n$.
\end{prop}
\begin{thm}[{\cite[Theorem 16.1]{M}}]
\label{matsu}
If $a_1, \ldots, a_n$ is a regular sequence, then so is $a_1^{v_1},\ldots,a_n^{v_n}$ for any positive integers $v_1,\ldots,v_n$.
\end{thm}
\begin{thm}[{\cite[Corollary of Theorem 16.3]{M}}]
\label{matsu2}
Let $A$ be a Noetherian ring and non-negatively graded.
If $a_1, \ldots, a_n$ is a regular sequence in $A$ and each $a_i$ is homogeneous of positive degree,
then any permutation of $a_1, \ldots, a_n$ is again a regular sequence.
\end{thm}
\begin{thm}[{cf. \cite[Theorem 5.5.1]{NS}}]
\label{NS}
Let $F$ be a field and $R = F[g_i \suchthat 1 \leq i \leq m]$
a non-negatively graded polynomial ring with $|g_i|>0$ for any $1 \leq i \leq m$.
Assume that $a_1, \ldots a_n$ is a regular sequence in $R$, which consists of homogeneous elements of positive degree.
Then the Poincar\'e series of $R/(a_i \suchthat 1 \leq i \leq n)$ is given as
\[
\frac{\prod_{i=1}^n (1-x^{|a_i|})}{\prod_{i=1}^m (1-x^{|g_i|})}.
\]
\end{thm}

\begin{proof}
For a non-negatively graded $F$-module $M$ of finite type, let $P(M,x)$ denote the Poincar\'e series of $M$,
namely
\[
P(M,x) = \sum_{n=0}^\infty (\dim_F M_n)x^n,
\]
where $M_n$ denotes the degree $n$ part of $M$.
Then obviously we have
\[
P(R,x) = \frac{1}{\prod_{i=1}^m (1-x^{|g_i|})}.
\]
Since $a_1, \ldots a_n$ is a regular sequence,
the multiplication by $a_i$ induces an injection
on a graded $F$-module $R/(a_1,\ldots,a_{i-1})$.
Therefore
\[
P(R/(a_1,\ldots,a_{i}),x) = (1-x^{|a_i|})P(R/(a_1,\ldots,a_{i-1}),x).
\]
The induction on $i$ completes the proof.
\end{proof}

\begin{proof}[Proof of Lemma \ref{free}]
Let $p$ be a prime number and
\[
M = (\Z[t_i, \gamma, \tau_i, \gamma_i, \omega \suchthat 1 \leq i \leq 4]
/\{ r'_1, R_i, r_{2i}, r_{12} \suchthat 1 \leq i \leq 4 \}),
\]
where $|t_i|=2$, $|\gamma_i|=2i$, $|\omega| = 4$.
We will show that the Poincar\'e series of
$M \otimes (\Z/p\Z)$ does not depend on $p$.
Then the graded $\Z$-module $M$ of finite type must be free.
The relations \eqref{r_2} and \eqref{r_4} say that
\[
\gamma_2 = \gamma_1(\gamma_1 +e_1(t)), \qquad \gamma_4 = -\biggr( \sum_{j=1}^{3} (-1)^j \gamma_j(\gamma_{4-j} +e_{4-j}(t)) -\omega \biggl),
\]
and then we can erase $\gamma_2$ and $\gamma_4$.
Let $R$ denote the polynomial ring $\Z[t_i, \gamma, \tau_i, \gamma_1, \gamma_3, \omega \suchthat 1 \leq i \leq 4]$,
$r_1'$, $R_i$'s, and $r_i$'s also denote the corresponding elements of $R$,
and $I$ denote the ideal generated by $\{ r'_1, R_i, r_6, r_8, r_{12} \suchthat 1 \leq i \leq 4 \}$ in $R$.
Since $M \cong R/I$, it is sufficient to compute the Poincar\'e series of $(R/I) \otimes (\Z/p\Z)$.

When $p=2$, we show that that the sequence
\[
r'_1, r_{12}, R_4, R_3, R_2, R_1, r_6, r_8
\]
is regular and compute the Poincar\'e series from this sequence.
In $(R/I)\otimes (\Z/2\Z)$, we have
\begin{align*}
& r'_1= e_1(t),\\
& R_1	= -(e_1(\tau)-e_1(t)), && R_2 = -(e_2(\tau)-e_2(t)) +(e_1(\tau)-e_1(t))e_1(t),\\
& R_3	= -(e_3(\tau)-e_3(t)), && R_4 = e_4(\tau) -e_4(t) -\omega,\\
&r_6 \equiv \gamma_3^2  \pmod{\gamma, e_i(t), \omega \suchthat 1 \leq i \leq 4},
&&r_8 \equiv \gamma_4^2 \equiv \gamma_2^2 \equiv \gamma_1^8  \pmod{\gamma, e_i(t),\omega \suchthat 1 \leq i \leq 4}.
\end{align*}

It is well known that the sequence of the elementary symmetric polynomials
\[
e_1(x), e_2(x), \ldots , e_n(x),
\]
that is, the sequence of the Chern classes 
is regular in $(\Z/p\Z)[x_i \suchthat 1\leq i \leq n]$ for any prime $p$.
Since a polynomial ring over a field is Noetherian,
by Theorem \ref{matsu2}, the sequence
\[
\gamma, e_1(t), e_2(t), e_3(t), e_4(t), \omega, e_4(\tau), e_3(\tau), e_2(\tau), e_1(\tau), \gamma_3^2, \gamma_1^8
\]
is regular in $R\otimes (\Z/2\Z)$.
We modify this sequence by Theorem \ref{matsu} and Proposition \ref{a+b} to obtain the following regular sequence.
\[
\gamma, r'_1, e_2(t), e_3(t), e_4(t), \omega^3, R_4, R_3, R_2, R_1, r_6, r_8
\]
Since $\rho^2 t_4 = -\gamma$ and $e_4(\rho t) = -e_4(t) -e_4(\rho^2 t) \equiv 0 \bmod{(\gamma, e_4(t))}$,
by Proposition \ref{a+b}
\[
\gamma, r'_1, e_2(t), e_3(t), e_4(t), r_{12}, R_4, R_3, R_2, R_1, r_6, r_8
\]
is a regular sequence.
Hence 
\[
r'_1, r_{12}, R_4, R_3, R_2, R_1, r_6, r_8
\]
is a regular sequence by Proposition \ref{ai}.
Finally, the Poincar\'e series of $(R/I) \otimes (\Z/2\Z)$
is calculated from the degrees of the generators and the relations by Theorem \ref{NS},
and we have
\begin{align*}
P(M \otimes (\Z/2\Z),x) = \frac{1}{(1-x^2)^4} (1 +x^8 +x^{16}) \prod_{i=1}^4 \frac{1-x^{4i}}{1-x^2}.
\end{align*}

Next let us consider the case where $p \geq 3$.
Let $e_1$, $e_2$, $e_3$, and $e_4$ be the left-hand sides of \eqref{e1}, \eqref{e2}, \eqref{e3}, and \eqref{e4} respectively, namely
\begin{align*}
e_1  &= e_1(\tau^2) -e_2(t^2),			&e_2  &= e_2(\tau^2) -e_2(t^2) -6\omega,\\
e_3  &= e_3(\tau^2) -e_2(t^2) -e_1(t^2)\omega,	&e_4  &= e_4(\tau^2) -e_4(t^2) +3\omega^2 -2(e_4(\rho t) -e_4(\rho^2 t))\omega.
\end{align*}
Recall that $e_1$, $e_2$, $e_3$, and $e_4$ are divided by $4$
to yield $r_2$, $r_4$, $r_6$, and $r_8$ respectively.
We have
\begin{align*}
M \otimes (\Z/p\Z)
 & \cong \Bigl(\Z[t_i, \gamma, \tau_i, \gamma_i, \omega \suchthat 1 \leq i \leq 4]
   /( r'_1, R_i, e_{2i}, r_{12} \suchthat 1 \leq i \leq 4 )\Bigr)\otimes (\Z/p\Z)\\
 & \cong \Big(\Z[t_i, \gamma, \tau_i, \omega \suchthat 1 \leq i \leq 4]
   /( r'_1, e_{2i}, r_{12} \suchthat 1 \leq i \leq 4 )\Big)\otimes (\Z/p\Z).
\end{align*}
since $2$ is invertible in $\Z/p\Z$.
We will show the sequence
\[
r'_1, r_{12}, e_8, e_6, e_4, e_2
\]
is a regular sequence.
It is well known that the sequence of elementary symmetric polynomial in $\{x_i^2\}_{i=1}^n$
\[
e_1(x^2), e_2(x^2), \ldots , e_n(x^2),
\]
that is, the sequence of the Pontryagin classes
is regular in $(\Z/p\Z)[x_i \suchthat 1\leq i \leq n]$ for any prime $p$.
By Theorem \ref{matsu2}, the sequence
\[
\gamma, e_1(t), e_2(t), e_3(t), e_4(t), \omega, e_4(\tau^2), e_3(\tau^2), e_2(\tau^2), e_1(\tau^2)
\]
is regular in $(\Z/p\Z)[t_i, \gamma, \tau_i, \omega \suchthat 1 \leq i \leq 4]$.
We modify this sequence by Theorem \ref{matsu} and Proposition \ref{a+b} to obtain the following regular sequence.
\[
\gamma, r'_1, e_2(t), e_3(t), e_4(t), r_{12}, e_8, e_6, e_4, e_2
\]
Hence
\[
r'_1, r_{12}, e_8, e_6, e_4, e_2
\]
is a regular sequence by Proposition \ref{ai}.
Therefore, by Theorem \ref{NS},
we have
\begin{align*}
P(M \otimes (\Z/p\Z),x) = \frac{1}{(1-x^2)^4} (1 +x^8 +x^{16}) \prod_{i=1}^4 \frac{1-x^{4i}}{1-x^2}.
\end{align*}
\end{proof}

\section{Proof of Corollary \ref{F4/T}}
\label{pf of Cor F4/T}

\begin{proof}
By the argument in Section \ref{intro} we have the isomorphisms
\begin{align*}
H^*(F_4/T) 	& \cong H^*_T(F_4/T)/(t_1, t_2, t_3, t_4, \gamma)\\
		& \cong \Z[\tau_i, \gamma_i, \omega \suchthat 1 \leq i \leq 4]/(Q_i, q_{2i}, q_{12} \suchthat 1 \leq i \leq 4),
\end{align*}
where
\begin{align*}
&Q_i = e_i(\tau) -2\gamma_i \:\: (i=1,2,3) , && Q_4 = e_4(\tau) -2\gamma_4 -\omega,\\
&q_2 = \gamma_2 - \gamma_1^2, && q_4 = \gamma_4 -2\gamma_1\gamma_3 +\gamma_2^2 - \omega, \\
&q_6 = 2\gamma_2\gamma_4 -\gamma_3^2 + \gamma_2 \omega, && q_8 = \gamma_4^2 +\gamma_4 \omega +\omega^2, \\
&q_{12} = \omega^3.
\end{align*}
We can regard $\gamma_2$ and $\gamma_4$ as dependent variables
by the relations $q_2$ and $q_4$.
Let $R$ be the polynomial ring $\Z[\tau_i, \gamma_1, \gamma_3, \omega \suchthat 1 \leq i \leq 4]$.
Then
\[
H^*(F_4/T) \cong R/(Q_i, q_6, q_8, q_{12} \suchthat 1 \leq i \leq 4).
\]
Obviously we have
\begin{align*}
& Q_i = -\overline{r}_i \:\: (i=1,2,3), && Q_4 \equiv \overline{r}_4 \pmod{Q_3}, \\
& q_6 \equiv \gamma_2e_4(\tau) -\gamma_3^2 = -\overline{r}_6 \pmod{Q_4}, && q_{12} = \overline{r}_{12}.
\end{align*}
Moreover we have
\begin{align*}
q_8 &= 4\gamma_1^2\gamma_3^2 -4\gamma_1^5\gamma_3 +\gamma_1^8 +3\omega(\omega +2\gamma_1\gamma_3) -3\gamma_1^4\omega \\
    &\equiv 8\gamma_1^4 \gamma_4 +4\gamma_1^4 \omega -4\gamma_1^5 \gamma_3 +\gamma_1^8 +3\omega(\omega +2\gamma_1\gamma_3) -3\gamma_1^4 \omega \pmod{q_6}\\
    &= 12\gamma_1^5 \gamma_3 -7\gamma_1^8 +9\gamma_1^4 \omega +3\omega(\omega +2\gamma_1\gamma_3) -3\gamma_1^4 \omega.
\end{align*}
On the other hand,
\begin{align*}
\overline{r}_8 	& = 3e_4(\tau) \gamma_1^4 -\gamma_1^8 +3\omega(\omega +e_3(\tau)\gamma_1)\\
			& \equiv 3(2\gamma_4 +\omega)\gamma_1^4 -\gamma_1^8 +3\omega(\omega +2\gamma_1\gamma_3) \pmod{Q_3,Q_4} \\
                  & = 12\gamma_1^5 \gamma_3 -7\gamma_1^8 +9\gamma_1^4 \omega +3\omega(\omega +2\gamma_1\gamma_3) -3\gamma_1^4 \omega.
\end{align*}
Hence $q_8 \equiv \overline{r}_8 \pmod{q_6, Q_3, Q_4}$.
Therefore
\[
H^*(F_4/T) \cong R/(Q_i, q_6, q_8, q_{12} \suchthat 1 \leq i \leq 4) \cong R/(\overline{r}_1, \overline{r}_2, \overline{r}_3, \overline{r}_4, \overline{r}_6, \overline{r}_8, \overline{r}_{12}).
\]
\end{proof}

\end{document}